\documentclass{amsart}

\title{The Donaldson-Thomas theory of $K3\times E$ via the topological vertex. }
\author{Jim Bryan}
\date{\today}
\address{
Department of Mathematics\\
University of British Columbia \\
Room 121, 1984 Mathematics Road  \\
Vancouver, B.C., Canada V6T 1Z2  
}

\usepackage{tikz}
\usepackage{verbatim}
\usetikzlibrary{calc,3d}

\definecolor{tealgreen}{HTML}{1B9E77}
\definecolor{orange}{HTML}{D95F02}
\definecolor{purple}{HTML}{7570B3}
\definecolor{pink}{HTML}{E7298A}
\definecolor{grassgreen}{HTML}{66A61E}
\definecolor{goldyellow}{HTML}{E6AB02}
\definecolor{brown}{HTML}{A6761D}
\definecolor{devilgray}{HTML}{666666}

\usepackage{amsmath}
\usepackage{amsmath,amsthm,amsfonts}
\usepackage{times}
\usepackage{verbatim}

\newtheorem{theorem}{Theorem}[section]

\newtheorem{proposition}[theorem]{Proposition}

\newtheorem{lemma}[theorem]{Lemma}

\newtheorem{conjecture}[theorem]{Conjecture}
\newtheorem{convention}[theorem]{Convention}
\newtheorem{question}[theorem]{Question}

\newtheorem{definition}[theorem]{Definition}

\newcommand{\CC} {{\mathbb C}}          
\newcommand{\ZZ} {{\mathbb Z}}		
\newcommand{\PP}{\mathbb{P}}

\renewcommand{\O}{\mathcal{O}}

\newcommand{\Hom}{\operatorname{Hom}}

\newcommand{\Sym}{\operatorname{Sym}}

\newcommand{\qtilde}{\tilde{q}\,}
\newcommand{\Hilb}{\operatorname{Hilb}}
\newcommand{\DT}{\operatorname{\mathsf{DT}}}
\newcommand{\DThat}{\widehat{\DT}}
\newcommand{\Vsf}{\mathsf{V}}
\newcommand{\Xhat}{\widehat{X}}

\newcommand{\fix}{\mathsf{fixed}}
\renewcommand{\Vert}{\mathsf{vert}}

\newcommand{\Var}{\operatorname{Var}}
\newcommand{\Spec}{\operatorname{Spec}}

\newcommand{\diag}{\mathsf{diag}}

\newcommand{\red}{\mathrm{red}}

\begin{document}

\maketitle 
\begin{abstract}
Oberdieck and Pandharipande conjectured \cite{Oberdieck-Pandharipande}
that the partition function of the Gromov-Witten/Donaldson-Thomas
invariants of $S\times E$, the product of a $K3$ surface and an
elliptic curve, is given by minus the reciprocal of the Igusa cusp
form of weight 10. For a fixed primitive curve class in $S$ of square
$2h-2$, their conjecture predicts that the corresponding partition
functions are given by meromorphic Jacobi forms of weight $-10$ and
index $h-1$. We calculate the Donaldson-Thomas partition function for
primitive classes of square -2 and of square 0, proving strong
evidence for their conjecture.

Our computation uses reduced Donaldson-Thomas invariants which are
defined as the (Behrend function weighted) Euler characteristics of the
quotient of the Hilbert scheme of curves in $S\times E$ by the action
of $E$. Our technique is a mixture of motivic and toric methods
(developed with Kool in \cite{Bryan-Kool}) which allows us to express
the partition functions in terms of the topological vertex and
subsequently in terms of Jacobi forms. We compute both versions of
the invariants: unweighted and Behrend function weighted Euler
characteristics. Our Behrend function weighted computation requires us
to assume Conjecture 18 in \cite{Bryan-Kool}.
\end{abstract}

\section{Overview}

Let $X=S\times E$ where $S$ is a $K3$ surface and $E$ is an elliptic
curve. In \cite{Oberdieck-Pandharipande}, Oberdieck and Pandharipande
conjectured that the partition function for the curve counting
invariants of $X$ is given by $-1/\chi _{10}$, minus the reciprocal of
the Igusa cusp form of weight 10\footnote{Since this paper was
originally written in 2015, Oberdieck and Pixton have recently given a
complete proof of this conjecture using methods (different from ours)
from both Donaldson-Thomas theory and Gromov-Witten theory
\cite{Oberdieck-Pixton}.}. The relevant curve counting invariants
include modified versions of Gromov-Witten invariants and stable pairs
invariants.  In this paper, we define modified Donaldson-Thomas
invariants of $X$. Our definition is given by taking the Euler
characteristic of the quotient of the Hilbert scheme of curves on $X$
by the action of the elliptic curve (we consider both the Behrend
function weighted and unweighted Euler characteristics). Our
invariants are expected to be equal to the invariants defined via
stable pairs in \cite{Oberdieck-Pandharipande}.

We employ an approach to computing these invariants which uses a
mixture of motivic and toric methods (technology developed with
M. Kool in \cite{Bryan-Kool}). We show that these methods yield
complete computations for the partition functions of $X$ in the case
where $S$ is $K3$ surface with a primitive curve class of square $-2$
or of square $0$ (assuming Conj. 18 \cite{Bryan-Kool} in the Behrend
function weighted case). The resulting partition functions are given
by the Jacobi forms $ F^{-2}\Delta ^{-1}$ and $-24\,\wp\,\Delta^{-1} $
respectively where $F$ is a Jacobi theta function, $\Delta $ is the
discriminant modular form, and $\wp $ is the Weierstrass $\wp$
function. This agrees with the prediction from by the
Oberdieck-Pandharipande conjecture thus proving their conjecture for
primitive curve classes in the $K3$ of square $-2$ or $0$.

Our general computational strategy is the following. Donaldson-Thomas
invariants are given by weighted Euler characteristics of Hilbert
schemes. We stratify the Hilbert scheme using the geometric support of
the curves and we compute Euler characteristics of strata
separately. Many of the strata acquire actions of $E$ or $\CC ^{*}$
(that were not present globally) and we restrict to the fixed point
loci. We are able to further stratify the fixed point loci and those
strata sometimes acquire further actions. Iterating this strategy, we
reduce the computation to subschemes which are formally locally given
by monomial ideals. These are counted using the topological
vertex. New identities for the topological vertex lead to closed
formulas. To incorporate the
Behrend function into this strategy, we must assume the
conjecture formulated in \cite[Conj.~18]{Bryan-Kool}. This is a
general conjecture regarding the behavior of the Behrend function at
subschemes given locally by monomial ideals.

\smallskip
\smallskip

\emph{Acknowledgements.} I'd like to thank George Oberdieck, Rahul
Pandharipande, and Yin Qizheng for invaluable discussions. I've also
benefited with discussions with Tom Coates, Sheldon Katz, Martijn
Kool, Davesh Maulik, Tony Pantev, Balazs Szendroi, Andras Szenes, and
Richard Thomas. The computational technique employed in this paper was
developed in collaboration with Martijn Kool whom I owe a debt of
gratitude. I would also like to thank the Institute for Mathematical
Research (FIM) at ETH for hosting my visit to Z\"urich, and for
Matematisk Institutt, UiO and Jan Christophersen for organizing the
2017 Abel Symposium.

\section{Definitions and conjectures}

Let $X $ be an arbitrary non-singular Calabi-Yau threefold over $\CC
$. One can define Donaldson-Thomas curve counting invariants by taking
weighted Euler characteristics of the Hilbert scheme of curves in
$X$. Let
\[
\Hilb ^{\beta ,n} (X) = \left\{Z\subset X:\,\,[Z]=\beta \in H_{2}
(X),\,\,n=\chi (\O _{Z}) \right\}
\]
be the Hilbert scheme of proper subschemes of $X$ with fixed homology
class and holomorphic Euler characteristic.

The Behrend function is a integer-valued constructible function
associated to any scheme over $\CC $. One can
define the Donaldson-Thomas invariants of $X$ by
\[
DT_{\beta ,n} (X)=\sum _{k\in \ZZ } k\cdot e\left(\nu ^{-1} (k)
\right)
\]
where
\[
\nu :\Hilb ^{\beta ,n} (X)\to \ZZ 
\]
is the Behrend function  \cite{Behrend-micro}.

It will be notationally convenient to treat an Euler characteristic
weighted by a constructible function as a Lebesgue integral, where the
measurable sets are constructible sets, the measurable functions are
constructible functions, and the measure of a set is given by its
Euler characteristic. In this language, one writes
\[
DT_{\beta ,n} (X) = \int _{\Hilb ^{\beta ,n} (X)} \nu \,de.
\]
For proper $X$, $DT_{\beta ,n} (X) $ as defined above is invariant
under deformations of $X$.

We now consider
\[
X=S\times E
\]
where $S$ is a non-singular $K3$ surface with a primitive curve class
$\beta $ of square
\[
\beta ^{2} = 2h-2.
\]
We call $h$ the \emph{genus} of the $K3$ surface. Let 
\[
\beta+dE\in H_{2} (X)
\]
denote the class $i_{S*} (\beta )+i_{E*} (d[E])$ where $i_{S}:S\to X$
and $i_{E}:E\to X$ are the inclusions obtained from choosing points
$s\in S$ and $e\in E$.

The Donaldson-Thomas invariants $DT_{\beta+dE,n} (X)$ are all
zero. This can be seen in two different ways:
\begin{enumerate}
\item The action of $E$ on $\Hilb ^{\beta+dE,n} (X)$ is fixed point
free, consequently its (Behrend function\footnote{The value of the
Behrend function at a closed point of a scheme only depends on the
local ring of that point, therefore the Behrend function of a scheme
is invariant under any group action.} weighted) Euler characteristic is
zero.
\item There exists deformations of $S$ which make $\beta $
non-algebraic. Under this deformation, the Hilbert scheme $\Hilb
^{\beta+dE,n} (X)$ becomes empty. Since $DT_{\beta+dE,n} (X)$ is
deformation invariant it must be zero.
\end{enumerate}

Remarkably, the above two issues can be solved simultaneously by
taking the weighted Euler characteristic of the quotient of the
Hilbert scheme.
\begin{definition}\label{defn: reduced DT invariants}
The \textbf{reduced} Donaldson-Thomas invariants of $X$ are defined by
\[
\DT _{\beta+dE,n} (X) = \int _{\Hilb ^{\beta+dE,n} (X)/E} \nu \,de
\]
where $\nu :\Hilb ^{\beta+dE,n} (X)/E\to \ZZ $ is the Behrend function
of the quotient. Note that we denote the reduced invariants with the
san serif font $\DT $, while the ordinary invariants have the ordinary
font $DT$.
\end{definition}

\begin{conjecture}\label{conj: reduced DT is invariant under deforms leaving beta algebraic}
The number $\DT _{\beta+dE,n} (X)$ is invariant under deformations
of $X$ which keep the class $\beta+dE$ algebraic.
\end{conjecture}

\emph{Proof sketch:} The Hilbert scheme $\Hilb ^{\beta+dE,n} (X)$
admits a $(-1)$-shifted symplectic structure coming from viewing it as
a moduli space of rank 1 sheaves on $X$ with trivialized determinant
\cite{PTVV}.  Taking the $(-1)$-symplectic quotient of the Hilbert
scheme by the action of $E$ yields a $(-1)$-symplectic space whose
underlying space is $\Hilb ^{\beta+dE,n} (X)/E$ (the moment map
affects the derived structure, but not the classical space). As with
any $(-1)$-shifted symplectic structure, this shifted symplectic
structure gives rise to a symmetric obstruction theory whose
associated virtual class has degree equal to the Behrend function
weighted Euler characteristic of underlying scheme. The effect of
taking the zeros of the moment map in the symplectic quotient
construction is to remove from the obstruction space those
obstructions to deforming the class $\beta $ to a non-algebraic
class. Note that these obstructions are dual to the deformations of a
subscheme given by the action of $E$. The resulting virtual class on
$\Hilb ^{\beta+dE,n} (X)/E$ should be invariant under deformations
preserving the algebraicity of $\beta $.

The analogous conjecture for reduced stable pairs invariants was
proved in \cite{ Oberdieck-reduced-stable-pair}.

Up to deformation, a curve class on a $K3$ surface is determined by
its square and divisibility, so by our assumption that $\beta $ is
primitive, it only depends on $h$ up to deformation. We thus
streamline the notation by writing:
\[
\DT _{h,d,n} (X):=\DT _{\beta+dE,n} (X)
\]
and we also write
\[
\Hilb ^{h,d,n} (X):=\Hilb ^{\beta+dE,n} (X).
\]

We also consider the related (but not \emph{a priori} deformation
invariant) quantity given by unweighted Euler characteristics.
\[
\DThat _{h,d,n} (X) = \int _{\Hilb ^{h ,d,n} (X)/E}1\,de.
\]

We define partition functions as follows
\begin{align*}
\DT (X) &= \,\,\sum _{h=0}^{\infty } \,\,\,\DT _{h} (X) \,\qtilde ^{h-1}\\
    	&= \sum _{\begin{smallmatrix} h,d\geq 0\\n\in \ZZ  \end{smallmatrix}} \DT _{h,d,n} (X) \,\qtilde ^{h-1}q^{d-1} (-p)^{n}\\
\DThat (X) &=\,\, \sum _{h=0}^{\infty }\,\,\, \DThat _{h} (X) \,\qtilde ^{h-1}\\
    	&= \sum _{\begin{smallmatrix} h,d\geq 0\\n\in \ZZ  \end{smallmatrix}} \DThat _{h,d,n} (X) \,\qtilde ^{h-1}q^{d-1} p^{n}
\end{align*}

We remark that our convention for the $\qtilde $ and $q$ variables is
the opposite from Oberdieck and Pandharipande's, however there is a
conjectural symmetry $\qtilde \leftrightarrow q$ and so this
difference should not be seen in the formulas. To be precise, the
Donaldson-Thomas version of Oberdieck and Pandharipande's conjecture
is the following.
\begin{conjecture}\label{conj: DT(X)=1/chi10}
Let $\chi _{10}$ be the Igusa cusp form of weight 10, then
\[
\DT (X) = -\frac{1}{ \chi _{10}}\,\,.
\]
Explicitly, we can write
\[
\chi _{10} \left(p,q,\qtilde  \right) = pq\qtilde \left(1-p^{-1} \right)^{2}\prod _{ n\in \ZZ}\prod _{ (d,h)> (0,0)} \left(1-p^{n}q^{d}\qtilde ^{h} \right)^{c\left(4dh-n^{2} \right)}
\]
where the integers $c (k)$ are given as the coefficients of $Z$, the
elliptic genus of the $K3$ surface:
\[
Z (p,q) = -24\wp F^{2} = \sum _{n\in \ZZ }\sum _{d\geq 0} c (4d-n^{2})p^{n}q^{d}.
\]
Here $F$ is a Jacobi theta function and $\wp $ is the Weierstrass $\wp
$-function, namely
\[
-F^{-2} = \frac{p}{(1-p)^{2}}\prod _{m=1}^{\infty }\frac{(1-q^{m})^{4}}{\left(1-p q^{m} \right)^{2}\left(1-p^{-1} q^{m} \right)^{2}}
\]
and 
\[
\wp =\frac{1}{12}+\frac{p}{(1-p)^{2}}+\sum _{d=1}^{\infty
}\left(\sum _{k|d}k (p^{k}+p^{-k}-2) \right)q^{d}.
\]
\end{conjecture}

Expanding $-\chi _{10}^{-1}$ as a series in $\qtilde $, one obtains
predictions for each $\DT _{h} (X)$ in terms of Jacobi forms of weight
-10 and index $h-1$ (see \cite[page~10]{Oberdieck-Pandharipande}). The
main result of this paper is the following theorem.

\begin{theorem}\label{thm: main theorem: DTh for h=0 and h=1}
The genus 0 and genus 1 partition functions for the unweighted
Donaldson-Thomas invariants are given by 
\begin{align*}
\DThat _{0} (X)
&= \frac{pq^{-1}}{(1-p)^{2}} \prod _{m=1}^{\infty } (1-q^{m})^{-20} (1-pq^{m})^{-2} (1-p^{-1}q^{m})^{-2}\\
\DThat _{1} (X) &= 24 q^{-1}\prod _{m=1}^{\infty } (1-q^{m})^{-24}\left(\frac{1}{12} + \frac{p}{(1-p)^{2}}+\sum _{d=1}^{\infty }\sum _{k|d}k\left(p^{k}+2 +p^{-k}\right) q^{d} \right)
\end{align*}
Moreover, assuming Conjecture \ref{conj: Behrend fnc conj} 
(i.e. \cite[Conj.~18]{Bryan-Kool}), the partition functions for the
Behrend function weighted Donaldson-Thomas invariants are given by
(note the sign differences with the above formulae, an overall sign on
each, and on the 2 within the sum)
\begin{align*}
\DT _{0} (X)
&= \frac{-pq^{-1}}{(1-p)^{2}} \prod _{m=1}^{\infty } (1-q^{m})^{-20} (1-pq^{m})^{-2} (1-p^{-1}q^{m})^{-2}\\
\DT _{1} (X) &= -24 q^{-1}\prod _{m=1}^{\infty } (1-q^{m})^{-24}\left(\frac{1}{12} + \frac{p}{(1-p)^{2}}+\sum _{d=1}^{\infty }\sum _{k|d}k\left(p^{k}-2 +p^{-k}\right) q^{d} \right)
\end{align*}

The above two formulas verify the Oberdieck-Pandharipande conjecture
for $K3$ surfaces with a primitive curve class of square $-2$ or
$0$. Namely, the series $\DT _{h} (X)$ for $h=0$ and $h=1$ are given
by the following Jacobi forms
\begin{align*}
\DT _{0} (X) &= \frac{1}{F^{2}\Delta },\\
\DT _{1} (X) &= -24\frac{\wp }{\Delta }.
\end{align*}

\end{theorem}

\section{Preliminaries and notation.}

Our aim is to compute $\DT _{h} (X) $ for $h=0$ and $h=1$. We begin by
computing $\DThat _{h} (X)$ and then discuss how to modify the
argument to include the Behrend function in section~\ref{sec: putting
in the Behrend function}.

Euler characteristic is motivic: it defines a homomorphism from $K_{0}
(\Var_{\CC })$, the Grothendieck group of varieties over $\CC $, to the
integers. We define 
\[
\Hilb ^{h,\bullet ,\bullet } (X)/E = \sum _{n,d} \left[\Hilb ^{h,d,n} (X)/E \right] p^{n} q^{d}
\]
which we regard as an element in $K_{0} (\Var _{\CC }) ((p))[[q]]$. We
will use this convention throughout:
\begin{convention}\label{bullet convention}
When an index is replaced by a bullet, we will sum over the index,
multiplying by the appropriate variable.
\end{convention}

We see that with our notation
\[
\DThat _{h} (X) = q^{-1} e\left(\Hilb ^{h,\bullet ,\bullet } (X)/E
\right).
\]

\begin{definition}
Let $p_{S}$ and $p_{E}$ be the projections of $X=S\times E$ onto each
factor and let $C\subset X$ be an irreducible curve. We say that $C$
is \textbf{vertical } if $p_{E}:C\to E$ is degree zero and we say $C$ is
\textbf{horizontal} if $p_{S}:C\to S$ is degree zero. If both maps are
of non-zero degree, we say $C$ is \textbf{diagonal}. See Figure~\ref{fig: diag, vert, and horiz curves}.
\end{definition}

We will assume that $X=S\times E$ where $S$ is generic among $K3$
surfaces admitting a primitive class $\beta $ of square $2h-2$. In
particular, $\beta $ is an irreducible class.

Since $\beta $ is an irreducible class, any subscheme $Z$
corresponding to a point in $\Hilb ^{h,d,n} (X)$ must have a unique
component $C_{0}\subset Z$ which is either a vertical or a diagonal
curve with all other curve components of $Z$ being
horizontal. Subschemes with $C_{0}$ diagonal cannot deform to
subschemes with $C_{0} $ horizontal and so we get a decomposition of
the Hilbert scheme into disjoint components corresponding to
subschemes with vertical and diagonal components respectively:
\[
\Hilb ^{h,d,n} (X) = \Hilb_{\Vert } ^{h,d,n} (X) \sqcup \Hilb_{\diag } ^{h,d,n} (X) 
\]
Diagonal curves do not appear in the $h=0$ case, but do occur for
$h\geq 1$.

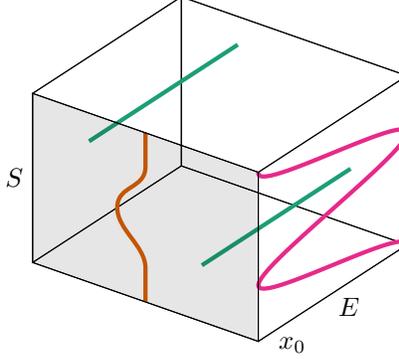
\begin{figure}
\caption{A vertical curve (orange) contained in the slice $S\times
\{x_{0} \}$ (light grey), a diagonal curve (pink), and two horizontal
curves (green).}\label{fig: diag, vert, and horiz curves}
\begin{tikzpicture}[
                    z  = {-15},
                    scale = 0.75]

\begin{scope}[yslant=-0.35,xslant=0]


\begin{scope} [canvas is yz plane at x=0]
\draw [black](0,0) rectangle (3,5);
\end{scope}
\begin{scope} [canvas is xz plane at y=0]
\draw [black](0,0) rectangle (4,5);
\end{scope}
\draw [black](0,0) rectangle (4,3);

\draw [ultra thick, tealgreen] (3,1,0)--(3,1,5) (1.0,2.5,0)-- (1.0,2.5,5);

\draw [ultra thick,orange] 
                   (2  ,0   ,5)
to [out=90,in=-90] (2  ,0.6 ,5)
to [out=90,in=-90] (1.5,1.5 ,5)
to [out=90,in=-90] (2  ,2.4 ,5)
to [out=90,in=-90] (2  ,3   ,5);

\node [left] at (0,1.5,5) {$S$};
\node [right] at (4.2,0,3) {$E$};
\node [right] at (4.2,0,5) {$x_{0}$};

\begin{scope} [canvas is yz plane at x=4]
\draw [black](0,0) rectangle (3,5);
\draw [pink, ultra thick, domain=0:3, samples=100] 
plot (\x ,{5*pow(sin(28.8*pi*\x),2)});
\end{scope}
\begin{scope} [canvas is xz plane at y=3]
\draw [black](0,0) rectangle (4,5);
\end{scope}
\draw [black](0,0,5) rectangle (4,3,5);
\draw [black,fill, opacity=0.1](0,0,5) rectangle (4,3,5);

\end{scope}
\end{tikzpicture}
\end{figure}

\section{Computing $\DThat _{h} (X)$ in the case $h=0$.}\label{sec: computing DThat for h=0}

We now consider the case where $h=0$. The $K3$ surface $S$ has a
single curve $C_{0}\cong \PP ^{1}$ in the class $\beta $. There are no
diagonal curves since such a curve would have geometric genus 0
but also admit a non-constant map to $E$.

\begin{figure}
\caption{Subschemes in $S\times E$ up to translation. Horizontal curves (pink) can have nilpotent thickenings (blue), and there can be embedded and floating points (gray). The unique vertical curve $C_{0}$ (green) lies in $S\times \{x_{0} \}$ and is generically reduced.}\label{fig: subschemes in SxE where S is genus 0}
\begin{tikzpicture}[
                    z  = {-15},
                    scale = 0.75,
embeddedpoint/.style={shade, ball color=#1}
]

\begin{scope}[yslant=-0.35,xslant=0]
  
\begin{scope} [canvas is yz plane at x=0]
\draw [black](0,0) rectangle (3,5);
\end{scope}
\begin{scope} [canvas is xz plane at y=0]
\draw [black](0,0) rectangle (4,5);
\end{scope}
\draw [black](0,0) rectangle (4,3);


\foreach \x in {2}
\foreach \y in {2.5}
{
\draw [ultra thick,pink] (\x ,\y,0)-- (\x ,\y,5);
\foreach \z in {0,0.1,...,5.1} 
    \draw [purple]
(\x ,\y ,\z ) -- ({\x +0.15*cos(24*pi*\z )},{\y+0.15*sin(24*pi*\z )},\z ) ;
}
\foreach \x in {3.2}
\foreach \y in {0.7}
{
\draw [ultra thick,pink] (\x ,\y,0)-- (\x ,\y,5);
\foreach \z in {0,0.1,...,5.1} 
    \draw [purple]
(\x ,\y ,\z ) -- ({\x +0.15*cos(36*pi*\z )},{\y+0.15*sin(36*pi*\z )},\z ) ;
}

\draw [ultra thick, pink] (2,0.5,0)--(2,0.5,5) (1.0,2.5,0)-- (1.0,2.5,5);

\draw [ultra thick,tealgreen] 
                   (2  ,0   ,5)
to [out=90,in=-90] (2  ,0.6 ,5)
to [out=90,in=-90] (1.5,1.5 ,5)
to [out=90,in=-90] (2  ,2.4 ,5)
to [out=90,in=-90] (2  ,3   ,5);


\foreach \p in {(1.6,1.2,5),(2.5,2,0), (3,1,0),(1,2.5,4.5),(2,2.5,5),(2,2.5,3)}
\shadedraw[ball color = gray] \p   circle (0.1);

\node [left] at (0,1.5,5) {$S$};
\node [right] at (4.2,0,3) {$E$};
\node [right] at (4.2,0,5) {$x_{0}$};
\node [below] at (2,0,5){$C_{0}$};

\draw [black](0,0,5) rectangle (4,3,5);

\begin{scope} [canvas is yz plane at x=4]
\draw [black](0,0) rectangle (3,5);
\end{scope}
\begin{scope} [canvas is xz plane at y=3]
\draw [black](0,0) rectangle (4,5);
\end{scope}

\end{scope}
\end{tikzpicture}

\end{figure}
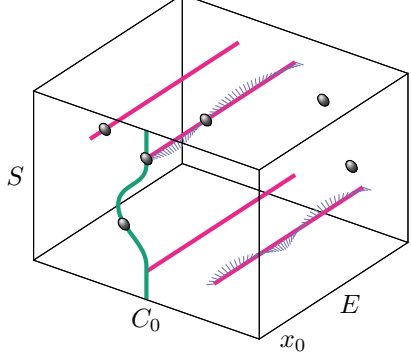

We fix a base point $x_{0}\in E$. We can fix a slice for the action of
$E$ on $\Hilb ^{0,d,n} (X)$ by requiring that the unique vertical
curve lies in $S\times \{x_{0} \}$. We denote the slice with the
subscript ``$\fix $''.
\[
\Hilb ^{0,d,n} (X)/E\cong \Hilb ^{0,d,n}_{\fix } (X)\subset \Hilb ^{0,d,n} (X).
\] 
The points in $\Hilb ^{0,d,n}_{\fix } (X)$ correspond to subschemes
$Z\subset X$ given by unions of the curve $C_{0}\times \{x_{0} \}$
with horizontal curves whose support is of the form
$\{\text{points}\times E \}$, but may have nilpotent thickenings. The
subscheme $Z$ also potentially has embedded points as well as zero
dimensional components away from the curve support (see
Figure~\ref{fig: subschemes in SxE where S is genus 0}).

As a consequence of the above geometric description, we see that any
such subscheme is a disjoint union of a subscheme of $\Xhat
_{C_{0}\times E}$, the formal neighborhood of $C_{0}\times E$ in $X$,
and $X- (C_{0}\times E)$. This leads to a decomposition of the Hilbert
scheme into strata given by products of Hilbert schemes of subschemes
of $\Xhat _{C_{0}\times E}$ and subschemes of $X- (C_{0}\times
E)$. Using our bullet convention, this can be efficiently expressed as
follows.
\begin{equation}\label{eqn: Hilb/E=Hilb(X-CxE)*Hilb(Xhat)}
\Hilb ^{0,\bullet ,\bullet } (X)/E = \Hilb ^{0,\bullet ,\bullet }_{\fix } \left(\Xhat
_{C_{0}\times E} \right)\cdot  \Hilb ^{0,\bullet ,\bullet } (X-
(C_{0}\times E)) 
\end{equation}
where as before the subscript ``$\fix $'' indicates that we are
restricting to the sublocus
\[
\Hilb ^{0,d,n}_{\fix } \left(\Xhat _{C_{0}\times E} \right)\subset
\Hilb ^{0,d,n}\left(\Xhat _{C_{0}\times E} \right)\subset \Hilb
^{0,d,n} (X)
\]
parameterizing subschemes where the unique vertical curve is
$C_{0}\times \{x_{0} \}$.

Note that $d$ (the degree in the $E$ direction) and $n$ (the
holomorphic Euler characteristic) are both additive under the disjoint
union which allows us to express the decomposition as a product of
Grothendieck group valued power series as above. Taking Euler
characteristics of the above series, we find
\begin{equation}\label{eqn: DT0=e(Hilb(Xhat))*e(Hilb(U))}
q\,\DThat _{0} (X) =  e\left(\Hilb ^{0,\bullet ,\bullet }_{\fix }\left(\Xhat
_{C_{0}\times E} \right) \right) \cdot  e\left(\Hilb ^{0,\bullet ,\bullet } (X-C_{0}\times
E) \right) .
\end{equation}

Note that the action of $E$ on $X-C_{0}\times E$ induces an action on
$\Hilb ^{0,d,n} (X-C_{0}\times E)$. This ``new'' $E$ action is
possible because the ``fixed'' condition lives entirely in the $\Hilb
^{0,d,n}_{\fix } \left(\Xhat _{C_{0}\times E} \right)$ factors (which
do not have $E$ actions).

The Euler characteristic of a scheme with a free $E$ action is trivial
and so
\[
e\left(\Hilb ^{0,d,n} (X-C_{0}\times E) \right) = e\left(\Hilb ^{0,d,n} (X-C_{0}\times E)^{E} \right) .
\]
The $E$-fixed locus $\Hilb ^{0,d,n} (X-C_{0}\times E)^{E}$
parameterizes subschemes which are invariant under the $E$
action. Such subschemes are of the form $Z\times E$ where $Z\subset
S-C_{0}$ is a zero-dimensional subscheme of length $d$. Such
subschemes have $n=\chi (\O _{Z\times E})=0$ and so
\begin{align}\label{eqn: e(Hilb (U))=prod(1-q^m)^(-22)}
\nonumber e\left(\Hilb ^{0,\bullet ,\bullet } (X-C_{0}\times E)^{E} \right) &= e\left(\sum _{d=0}^{\infty }\Hilb ^{d} (S-C_{0})\,q^{d} \right)\\
&=\prod _{m=1}^{\infty }\left(1-q^{m} \right)^{-22}.
\end{align}
Here we have used G\"ottsche's formula for the Euler characteristics
of Hilbert schemes of points of surfaces; the 22 appearing in the
exponent is the Euler characteristic of the surface $S-C_{0}$.

To compute $e\left(\Hilb _{\fix }^{0,\bullet,\bullet } (\Xhat
_{C_{0}\times E}) \right)$, we begin by noting that there is a
morphism
\[
\rho _{d} : \Hilb _{\fix }^{0,d,n} (\Xhat _{C_{0}\times E}) \to \Sym^{d} (C_{0})
\]
given by the intersection (with multiplicity) of the horizontal
components of a curve with the vertical curve $C_{0}$. In other words,
a scheme whose curve support is $C_{0}\cup _{i} ( y_{i}\times E)$ with
multiplicity $a_{i}$ along $y_{i}\times E$ is mapped to $\sum
_{i}a_{i}y_{i}\in \Sym ^{d} (C_{0})$ (see Figure~\ref{fig: map
Hilb(Xhat)-->SymC0}).

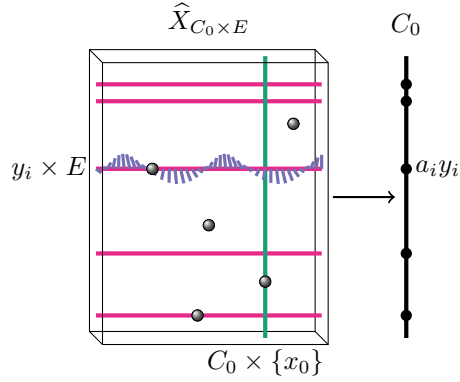
\begin{figure}
\caption{The map $\rho _{d} : \Hilb _{\fix }^{0,d,n} (\Xhat _{C_{0}\times E}) \to \Sym^{d} (C_{0})$ records the location and multiplicity of the horizontal curve components.}\label{fig: map Hilb(Xhat)-->SymC0}
\begin{tikzpicture}[scale=0.75]
\begin{scope}[rotate=90]
\draw (0,0,-0.3) rectangle (5,4,-0.3);
\draw (0,0,-.3)-- (0,0,.3)
      (5,0,-.3)-- (5,0,.3)
      (0,4,-.3)-- (0,4,.3)
      (5,4,-.3)-- (5,4,.3);
\draw [thick,->] (2.5,-0.2)-- (2.5,-1.3);
\draw [ultra thick] (0,-1.5)-- (5,-1.5);
\draw (5.2,2)node [above]{$\Xhat _{C_{0}\times E}$};
\draw (5.2,-1.5)node[above]{$C_{0}$};
\draw [ultra thick, pink](3,0,0)-- (3,4,0);
\draw (3,4.0,0) node[left]{$y_{i}\times E$};
\fill (3,-1.5) circle[radius=.1] node[right]{$a_{i}y_{i}$};
\foreach \x in {0.4,1.5,4.2,4.5}
{\draw [ultra thick, pink](\x ,0,0)-- (\x ,4,0);
\fill (\x ,-1.5) circle[radius=.1] ;};

\draw [tealgreen,ultra thick] (0,1)node[below,black]{$ C_{0}\times \{x_{0} \}$}-- (5,1);

\foreach \x in {3}
{
\draw [ultra thick,pink] (\x ,0)-- (\x ,4);
\foreach \y in {0,0.1,...,4.1} 
    \draw [very thick, purple]
(\x ,\y ,0 ) -- ({\x +0.25*cos(64*pi*\y )},\y,{0.25*sin(64*pi*\y )} ) ;
}

\foreach \p in {(1,1),(0.4,2.2),(2,2), (3,3),(3.8,.5)}
\shadedraw[ball color = gray] \p   circle (0.1);

\draw (0,0,.3) rectangle (5,4,.3);

\end{scope}
\end{tikzpicture}
\end{figure}
We may compute the Euler characteristic of $\Hilb _{\fix }^{0,\bullet
,\bullet } (\Xhat _{C_{0}\times E})$ by computing the Euler
characteristic of $\Sym ^{d} (C_{0})$, weighted by the constructible
function given by the Euler characteristic of the fibers of $\rho
_{d}$. In other words,

\begin{align*}
e\left(\Hilb _{\fix }^{0,d,n} (\Xhat _{C_{0}\times E}) \right) &= \int _{\Hilb _{\fix }^{0,d,n} (\Xhat _{C_{0}\times E}) } 1\,\, de\\
&=\int _{\Sym ^{d}C_{0}} (\rho _{d})_{*} (1) \,\,de.
\end{align*}
Writing 
\[
\Sym ^{\bullet } C_{0} = \sum _{d=0}^{\infty }\Sym
^{d}C_{0}\,q^{d}
\]
and extending the integration to the $\bullet $
notation in the obvious way, we get
\begin{equation}\label{eqn: e(Hilb(Xhat))=int rho* de}
e\left(\Hilb _{\fix }^{0,\bullet,\bullet}  (\Xhat _{C_{0}\times E}) \right) =\int _{\Sym ^{\bullet } C_{0}} \rho_{*} (1) \,\,de
\end{equation}
where the constructible function $\rho _{*} (1) $ takes values in $\ZZ
((p))$ and is given by
\[
\rho _{*} (1) (\sum _{i}a_{i}y_{i}) = e\left(\rho ^{-1} (\sum
_{i}a_{i}y_{i}) \right).
\]

We will prove that $\rho _{*} (1)$ only depends on the multiplicities
of the points in the symmetric product, not their location. 

\begin{proposition}\label{prop: rho* (1)= (1-p)^2/p*prod_i F (a_i)}
There exists a universal series $F (a)\in \ZZ [[p]]$ such that the
constructible function $\rho _{*} (1)$ is given by
\[
\rho _{*} (1) \left(\sum a_{i}y_{i} \right) =
\frac{p}{(1-p)^{2}}\prod _{i}F (a_{i}).
\]
\end{proposition}
Deferring the proof of the proposition for the moment, we apply the
following lemma regarding weighted Euler characteristics of symmetric
products.

\begin{lemma}\label{lem: euler chars of sym}
Let $T$ be a scheme and let $\Sym ^{\bullet } (T)=\sum _{d=0}^{\infty
}\Sym ^{d} (T)\, q^{d}$. Suppose that $G$ is a constructible function
on $\Sym ^{d} (T)$ of the form $G (\sum _{i}a_{i}y_{i})=\prod _{i}g
(a_{i})$ where by convention $g (0)=1$. Then
\[
\int _{\Sym ^{\bullet }T} G\,\, de = \left(\sum _{a=0}^{\infty }g
(a)\, q^{a} \right)^{e (T)}.
\]
\end{lemma}
This lemma is a consequence of the fact that symmetric products define
a pre-lambda ring structure on the Grothendieck group of varieties and
the Euler characteristic homomorphism is compatible with that
structure. An elementary proof is given in \cite{Bryan-Kool}.

Applying Lemma~\ref{lem: euler chars of sym} to Proposition~\ref{prop:
rho* (1)= (1-p)^2/p*prod_i F (a_i)} and combining with equations
\eqref{eqn: DT0=e(Hilb(Xhat))*e(Hilb(U))}, \eqref{eqn: e(Hilb
(U))=prod(1-q^m)^(-22)}, and \eqref{eqn: e(Hilb(Xhat))=int rho* de}
and we see that
\begin{equation}\label{eqn: DT0 =p/(1-p^2) (F series)^2*prod (1-q^m)^{-22}}
q\,\DThat _{0} (X) = \frac{p}{(1-p)^{2}}\left(\sum _{a=0}^{\infty }F
(a)\,q^{a} \right)^{2}\cdot \prod _{m=1}^{\infty } (1-q^{m})^{-22}.
\end{equation}

To finish the computation of $\DThat _{0} (X)$, we need to prove
Proposition~\ref{prop: rho* (1)= (1-p)^2/p*prod_i F (a_i)} and compute
the series $\sum _{a}F (a)q^{a}$.

\subsection{Proof of Proposition~\ref{prop: rho* (1)= (1-p)^2/p*prod_i
F (a_i)} and the computation of $\sum _{a}F (a)q^{a}$.}\label{subsec:
computation of F(a), smooth fiber contributions} The fiber $\rho ^{-1}
(\sum a_{i}y_{i})$ parameterizes subschemes supported on $\Xhat
_{C_{0}\times E}$ which have fixed curve support
\[
C_{0}\times x_{0} \cup _{i} \{y_{i} \}\times E
\]
where the multiplicity of the subscheme along $\{y_{i} \}\times E$ is
$a_{i}$. Such a subscheme is \emph{uniquely} determined\footnote{This
follows from fpqc descent since the set $U$ and the sets $\Xhat
_{\{y_{i} \}\times E}$ form a fpqc cover. Since $C_{0}\times x_{0}$ is
reduced there are no conditions on the overlaps of the cover. Thus the
subscheme is \emph{uniquely} determined by its restriction to the
cover.} by its restriction to the formal neighborhoods $\Xhat
_{\{y_{i} \}\times E}$ and their complement $U$ in $\Xhat
_{C_{0}\times E}$. The resulting stratification leads to a product
decomposition for the Grothendieck group valued power series $\rho
^{-1} (\sum a_{i}y_{i})$ giving the product formula in
Proposition~\ref{prop: rho* (1)= (1-p)^2/p*prod_i F (a_i)}. The factor
$p (1-p)^{-2}$ comes from the contribution of $U$ and it is the series
for the Hilbert scheme of subschemes of $\Xhat _{C_{0}\times E}$ with
fixed curve support $C_{0}\times x_{0}$ (no curves in the $E$
direction). The moduli for this Hilbert scheme comes from floating
points and embedded points (see \cite{Bryan-Kool} for details).

The series $F (a)$ is given by
\[
F (a) = (1-p)\cdot e\left(\Hilb ^{0,a,\bullet }\left(\Xhat _{\{y_{i}\}\times E } \right) \right)
\]
where 
\[
\Hilb ^{0,a,\bullet }\left(\Xhat _{\{y_{i}\}\times E} \right)\subset
\Hilb ^{0,a,\bullet } (X)
\]
is the locus parameterizing subschemes $Z$ whose curve support is
given by the union of $C_{0}\times \{x_{0} \}$ and an $a$-fold
thickening of $\{y_{i} \}\times E$ and such that all embedded points
of $Z$ are supported on $\Xhat _{\{y_{i} \}\times E}$. The prefactor
$(1-p)$ comes from the contribution of the complement $U$: the overall
contribution of $U$ is given by $p (1-p)^{-2+l}$ where $l$ is the
number of $y_{i}$'s and so we have redistributed the $l$ copies of
$(1-p) $ into the $F (a_{i})$ factors.

Since 
\[
\Xhat _{\{y_{i} \}\times E} \cong \Spec\left(\CC [[u,v]] \right)\times
E,
\]
we get an action of $\left(\CC ^{*} \right)^{2}$ on the corresponding
Hilbert scheme. Only the $\left(\CC ^{*} \right)^{2}$ fixed points
contribute to the Euler characteristic so
\begin{align*}
F (a) &= (1-p)\cdot e\left(\Hilb ^{0,a,\bullet }\left(\Xhat _{\{y_{i}\}\times E } \right)^{(\CC ^{*})^{2}} \right)\\
&= (1-p)\sum _{\alpha \vdash a} e\left(\Hilb ^{0,\alpha ,\bullet
}\left(\Xhat _{\{y_{i} \}\times E} \right) \right)
\end{align*}
where $\Hilb ^{0,\alpha ,\bullet }\left(\Xhat _{\{y_{i} \}\times E}
\right)$ parameterizes subschemes whose curve component is the
\emph{unique } curve given by the union of $C_{0}\times \{x_{0} \}$
and $Z_{\alpha }\times E$ where $Z_{\alpha }\subset \Spec (\CC
[[u,v]])$ is the length $a$ subscheme given by the monomial ideal
determined\footnote{i.e. identifying the partition $\alpha $ with its
Ferrer's diagram $\alpha \subset (\ZZ _{\geq 0})^{2}$, the ideal of
$Z_{\alpha } $ is generated by the monomials $u^{i}v^{j}$ where
$(i,j)\not\in \alpha$.} by the partition $\alpha \vdash a$.

To compute $e \left(\Hilb ^{0,\alpha ,\bullet }\left(\Xhat _{\{y_{i}
\}\times E} \right) \right)$ we can now integrate over the fibers of
the constructible morphism
\[
\sigma : \Hilb ^{0,\alpha ,\bullet }\left(\Xhat _{\{y_{i}
\}\times E} \right) \to \Sym ^{\bullet }E
\]
which is defined by recording the length and locations of the embedded
points. We thus get
\[
\int _{\Hilb ^{0,\alpha ,\bullet }\left(\Xhat _{\{y_{i}
\}\times E} \right)} \,\,de = \int _{\Sym ^{\bullet }E} \sigma _{*} (1)\,de .
\]
The constructible function $\sigma _{*} (1)$ is a product of local
contributions which only depend on the length of the embedded point
and whether or not the location of the embedded point is $x_{0} $ or
not (recall that $x_{0}$ is where the curve $C_{0}\times \{x_{0} \}$
is attached to the curve $Z_{\alpha }\times E$ ).  Writing the series
for the local contributions at $x_{0}$ and at the general point as
$\Vsf _{\emptyset (1)\alpha } (p)$ and $\Vsf _{\emptyset \emptyset
\alpha } (p)$ respectively, and applying Lemma~\ref{lem: euler chars
of sym} we get
\begin{align*}
\int _{\Sym ^{\bullet }E} \sigma _{*} (1)\,de &= \left(\Vsf _{\emptyset
(1)\alpha } (p) \right)\cdot \left(\Vsf _{\emptyset \emptyset
\alpha } (p) \right)^{e (E-x_{0})} \\
&= \frac{\Vsf _{\emptyset
(1)\alpha } (p)}{\Vsf _{\emptyset \emptyset \alpha } (p)}.
\end{align*}

The above naming of the local contributions is not accidental --- the
generating functions for the contributions are given by the
topological vertex. In general, the topological vertex $\Vsf _{\mu
_{1}\mu _{2}\mu _{3}} (p)$ can be defined as the generating function
of the Euler characteristics of the Hilbert schemes $\Hilb ^{n}
\left(\widehat{\CC }_{0}^{3},\{\mu _{1},\mu _{2},\mu _{3} \} \right)$,
which by definition parameterize subschemes of $\CC ^{3} $ given by
adding at the origin a length $n$ embedded point to the fixed curve
$Z_{\mu _{1}}\cup Z_{\mu _{2}}\cup Z_{\mu _{3}}$. Here $Z_{\mu _{i}}$
is supported on the $i$th coordinate axis and given by the monomial
ideal determined by the partition $\mu _{i}$ in the transverse
directions. Because $\left(\CC ^{*} \right)^{3}$ acts on these Hilbert
schemes, their Euler characteristics can be computed by counting
$\left(\CC ^{*} \right)^{3}$ fixed points, namely monomial
ideals. This leads to the combinatorial interpretation of $\Vsf _{\mu
_{1}\mu _{2}\mu _{3}} (p)$ --- it is the generating function for the
number of 3D partitions with asymptotic legs given by $\{\mu _{1},\mu
_{2},\mu _{3} \}$.

We thus get the following formula
\[
\sum _{a=0}^{\infty }F (a)q^{a} = \sum _{\alpha } \,q^{|\alpha |}
(1-p)\,\frac{\Vsf _{\emptyset (1)\alpha } (p)}{\Vsf _{\emptyset
\emptyset \alpha } (p)}
\]
which completes the proof of Proposition~\ref{prop: rho* (1)=
(1-p)^2/p*prod_i F (a_i)}.

\begin{lemma}\label{lem: formula for sum F (a)q^a}
The generating function for the universal series $F (a)$ is given by
the following formula
\[
\sum _{a=0}^{\infty }F (a)q^{a} = \prod _{m=1}^{\infty } \frac{(1-q^{m})}{(1-pq^{m}) (1-p^{-1}q^{m})} .
\]
\end{lemma}
\begin{proof}
Using the Okounkov-Reshetikhin-Vafa formula for the vertex
\cite[eqn~3.20]{Okounkov-Reshetikhin-Vafa}, the sum
\[
\sum _{\alpha } \,q^{|\alpha |}
(1-p)\,\frac{\Vsf _{\emptyset (1)\alpha } (p)}{\Vsf _{\emptyset
\emptyset \alpha } (p)}
\]
can be expressed as the trace of a certain natural operator on Fock
space. It can be evaluated explicitly by a theorem of Bloch-Okounkov
\cite[Thm~6.5]{Bloch-Okounkov}. The result is the product formula
given by the lemma.  See \cite{Bryan-Kool-Young} for details.
\end{proof}

Substituting the formula of the lemma into equation~\eqref{eqn: DT0
=p/(1-p^2) (F series)^2*prod (1-q^m)^{-22}} we get
\[
\DThat _{0} (X) = \frac{pq^{-1}}{(1-p)^{2}}\prod _{m=1}^{\infty } (1-q^{m})^{-20} (1-pq^{m})^{-2} (1-p^{-1}q^{m})^{-2}
\]
which proves the $g=0$ formula in Theorem~\ref{thm: main theorem: DTh
for h=0 and h=1}, assuming that we can show 
\[
\DT _{0} (X) = -\DThat _{0} (X).
\]
We will address this issue in section~\ref{sec: putting in
the Behrend function}.

\section{The case of $h=1$. }
   
We now consider the case where $S$ has a primitive curve class $\beta
$ with $\beta ^{2}=0$. Such $K3$ surfaces are elliptically fibered
with fiber class $\beta $. By our genericity assumption, we may assume
that the elliptic fibration $\pi :S\to \PP ^{1}$ has 24 singular
fibers, all of which are nodal, and we will further assume that the
fibration has a section (see figure~\ref{fig: elliptic fibration}).
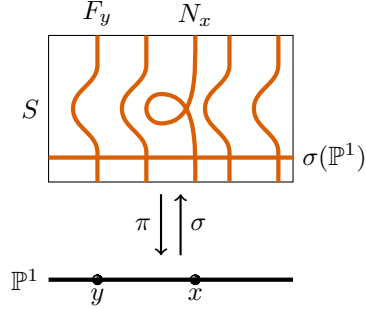
\begin{figure}
\caption{$\pi :S\to \PP ^{1}$ is a elliptic fibration with 24 nodal
fibers and a section $\sigma $. Figure depicts a smooth fiber $F_{y}$
over a point $y\in \PP^{1}$ and a nodal fiber $N_{x}$ over a point
$x\in \PP ^{1}$.}\label{fig: elliptic fibration}

\begin{tikzpicture}[scale=0.65]

\draw (0,0) rectangle (5,3);

\foreach \x in {0,-1,1.7,2.7}
\draw [ultra thick,orange] 
                   (2+\x   ,0  )
to [out=90,in=-90] (2+\x   ,0.6 )
to [out=90,in=-90] (1.5+\x ,1.5  )
to [out=90,in=-90] (2+\x   ,2.4  )
to [out=90,in=-90] (2+\x   ,3    );
\node[above] at (1,3) {$F_{y}$};
\draw[ball color=black] (1,-2)node[below]{$y$}   circle (0.1);

\draw [ultra thick,orange] 
                    (3   ,0    ) 
to [out=90,in=0]    (2.3 ,1.8  ) 
to [out=180,in=90]  (2   ,1.5  ) 
to [out=270,in=180] (2.3 ,1.2  ) 
to [out=0,in=270]   (3   ,3    );
\node[above] at (3,3) {$N_{x}$};
\draw[ball color=black] (3,-2)node[below]{$x$}   circle (0.1);

\draw [ultra thick, orange] (0,.5)-- (5,.5);
\node[right] at (5,.5) {$\sigma (\PP ^{1})$};
\draw [thick,->] (2.3,-.25)--node[left]{$\pi $} (2.3,-1.5);
\draw [thick,->] (2.7,-1.5)--node[right]{$\sigma  $} (2.7,-.25);
\draw [ultra thick] (0,-2)--(5,-2);
\node [left] at (0,-2) {$\PP ^{1}$};
\node [left] at (0,1.5) {$S$};

\end{tikzpicture}
\end{figure}

Recall that the Hilbert scheme decomposes into a disjoint union
\[
\Hilb ^{1,d,n} (X) = \Hilb _{\Vert }^{1,d,n} (X) \sqcup \Hilb _{\diag
}^{1,d,n} (X).
\]
We can fix a slice for the $E$ action on $\Hilb _{\Vert }^{1,d,n} (X)$
by requiring that the unique vertical curve lies in $S\times \{x_{0}
\}$. In the case where the subscheme has a diagonal curve, we require
that the diagonal curve intersects the slice $S\times \{x_{0} \}$
somewhere on the section. Denoting the above conditions with the
subscript $\fix $, we get
\[
\Hilb ^{1,d,n} (X)/E\cong \Hilb ^{1,d,n}_{\Vert ,\fix } (X)\sqcup
\Hilb ^{1,d,n}_{\diag ,\fix } (X)
\]
and so 
\[
q\,\DThat _{1} (X) = e\left(\Hilb ^{1,\bullet ,\bullet }_{\Vert ,\fix
} (X) \right)+e\left(\Hilb ^{1,\bullet ,\bullet }_{\diag ,\fix } (X)
\right).
\]

We get a map
\[
\tau :\Hilb ^{1,\bullet ,\bullet }_{\Vert ,\fix } (X)\to \PP ^{1}
\]
induced by the elliptic fibration $S\to \PP ^{1}$ since each subscheme
parameterized by $\Hilb ^{1,\bullet ,\bullet }_{\Vert ,\fix } (X)$
has a unique vertical curve which is a fiber curve. Let $F_{y}$ denote
the fiber of $S\to \PP ^{1}$ over $y$. Let
\[
\Hilb _{F_{y}}^{1,d,n} (X)\subset \Hilb _{\Vert ,\fix
}^{1,d,n} (X)
\]
denote the sublocus which parameterizes subschemes whose unique
vertical component is $F_{y}\times \{x_{0} \}$.

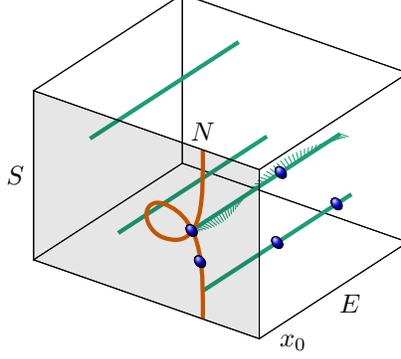
\begin{figure}
\caption{A configuration which includes a thickened horizontal curve
(green) attached to the node of a nodal vertical curve (orange). For
the contribution to be non-zero, embedded points (blue) must occur
along horizontal curves attached to the vertical curve or on the
vertical curve.  }\label{fig: h=1 configurations}
\begin{tikzpicture}[
                    z  = {-15},
                    scale = 0.75]

\begin{scope}[yslant=-0.35,xslant=0]


\begin{scope} [canvas is yz plane at x=0]
\draw [black](0,0) rectangle (3,5);
\end{scope}
\begin{scope} [canvas is xz plane at y=0]
\draw [black](0,0) rectangle (4,5);
\end{scope}
\draw [black](0,0) rectangle (4,3);

\foreach \x in {2.8}
\foreach \y in {1.5}
{
\draw [ultra thick,tealgreen] (\x ,\y,0)-- (\x ,\y,5);
\foreach \z in {0,0.1,...,5.1} 
    \draw [thin, tealgreen]
(\x ,\y ,\z ) -- ({\x +0.15*cos(24*pi*\z )},{\y+0.15*sin(24*pi*\z )},\z ) ;
}

\draw [ultra thick, tealgreen] (3,.5,0)--(3,.5,5) (1.0,2.5,0)-- (1.0,2.5,5) (1.5,1,0)-- (1.5,1,5);

\draw [ultra thick,orange] 
                    (3   ,0   ,5) 
to [out=90,in=0]    (2.3 ,1.8 ,5) 
to [out=180,in=90]  (2   ,1.5 ,5) 
to [out=270,in=180] (2.3 ,1.2 ,5) 
to [out=0,in=270]   (3   ,3   ,5);


\foreach \p in {
(2.8,1.5,5), 
(3,.5,.5), 
(3,.5,2.5),
(2.95,1,5),
(2.8,1.5,2)
}
\shadedraw[ball color = blue] \p   circle (0.1);

\node [above] at (3,3,5) {$N$};

\node [left] at (0,1.5,5) {$S$};
\node [right] at (4.2,0,3) {$E$};
\node [right] at (4.2,0,5) {$x_{0}$};

\begin{scope} [canvas is yz plane at x=4]
\draw [black](0,0) rectangle (3,5);
\end{scope}
\begin{scope} [canvas is xz plane at y=3]
\draw [black](0,0) rectangle (4,5);
\end{scope}
\draw [black](0,0,5) rectangle (4,3,5);
\draw [black,fill, opacity=0.1](0,0,5) rectangle (4,3,5);

\end{scope}
\end{tikzpicture}
\end{figure}

We will see below that the Euler characteristic of $\Hilb
_{F_{y} }^{1,d,n} (X)$ only depends on the topological
type of the fiber, i.e. whether it is smooth or nodal. We write a
generic smooth fiber as $F$ and any nodal fiber as
$N$. Integrating over the fibers of $\tau $, we get
\[
e\left(\Hilb ^{1,\bullet ,\bullet }_{\Vert ,\fix } (X) \right) =
-22e\left(\Hilb ^{1,\bullet ,\bullet }_{F} (X)
\right) + 24e\left(\Hilb ^{1,\bullet ,\bullet }_{N}
(X) \right)
\]
where the $-22$ is $e (\PP ^{1}-24\text{pts})$. See figure~\ref{fig:
h=1 configurations} for a depiction of a curve configuration
corresponding to a point in $\Hilb ^{1,\bullet ,\bullet }_{N} (X)$.

The computation of $e\left(\Hilb ^{1,\bullet ,\bullet }_{F} (X)
\right)$ and $e\left(\Hilb ^{1,\bullet ,\bullet }_{N} (X) \right)$
follows the same strategy as the computation of $e\left(\Hilb
^{0,\bullet ,\bullet }_{\fix } (X) \right)$ done in section~\ref{sec:
computing DThat for h=0}. We use the product decompositions
\begin{align}\label{eqn: Hilb/E=Hilb(X-FxE)*Hilb(Xhat)}
\Hilb ^{1,\bullet ,\bullet }_{F} (X)&= \Hilb _{F}^{1,\bullet ,\bullet
}\left(\Xhat _{F\times E} \right)\cdot \Hilb ^{1,\bullet ,\bullet
}\left(X-F\times E \right)\\ \label{eqn: Hilb/E=Hilb(X-NxE)*Hilb(Xhat)}
\Hilb ^{1,\bullet ,\bullet }_{N} (X)&= \Hilb _{N}^{1,\bullet ,\bullet
}\left(\Xhat _{N\times E} \right)\cdot \Hilb ^{1,\bullet ,\bullet
}\left(X-N\times E \right)
\end{align}
and we use the extra $E$ actions on the second factors to deduce
\begin{align*}
e\left(\Hilb _{F}^{1,\bullet ,\bullet } (X) \right)&= e\left(\Hilb _{F}^{1,\bullet ,\bullet }\left(\Xhat _{F\times E} \right) \right)\cdot \prod _{m=1}^{\infty } (1-q^{m})^{-24}\\
e\left(\Hilb _{N}^{1,\bullet ,\bullet } (X) \right)&= e\left(\Hilb
_{N}^{1,\bullet ,\bullet }\left(\Xhat _{N\times E} \right)
\right)\cdot \prod _{m=1}^{\infty } (1-q^{m})^{-23}
\end{align*}
where $24 = e (S-F)$ and $23=e (S-N)$.

Proceeding as we did in section~\ref{sec: computing DThat for h=0}, we
use the maps
\begin{align*}
\rho :\Hilb _{F}^{1,\bullet ,\bullet }\left(\Xhat _{F\times E} \right)
&\to \Sym ^{\bullet } (F)\\
\rho :\Hilb _{N}^{1,\bullet ,\bullet }\left(\Xhat _{N\times E} \right)
&\to \Sym ^{\bullet } (N)
\end{align*}
which record the location and multiplicity of the horizontal
components. The argument proceeds exactly as it did in
section~\ref{sec: computing DThat for h=0} with $F$ and $N$ playing
the role of $C_{0}$.

The result for the smooth fiber case is the following:
\begin{align}\label{eqn: e(Hilb(Xhat_{FxE}))=1}
 \nonumber
\int _{\Hilb _{F}^{1,\bullet ,\bullet }\left(\Xhat _{F\times E} \right)} de &=\int _{\Sym ^{\bullet } (F)}\rho _{*} (1)\,de\\
&=\left(p^{1/2} (1-p)^{-1} \right)^{e (F)}\cdot \left(\sum _{a=0}^{\infty }F (a)q^{a} \right)^{e (F)}\\
 \nonumber
&=1.
\end{align}
This result comports with the heuristic that $F$ acts on $\Xhat
_{F\times E}$ and hence on $\Hilb _{F}^{1,\bullet ,\bullet
}\left(\Xhat _{F\times E} \right)$ and so the Euler characteristic is
0 except for the unique $F$-fixed subscheme, i.e. the subscheme
consisting of just the curve $F\times \{x_{0} \}$ with no added
horizontal components or embedded points. However, this is only a
heuristic: $F$ does \emph{not} act algebraically on the formal
neighborhood $\Xhat _{F\times E}$ since the elliptic fibration is not
isotrivial\footnote{However, see section~\ref{sec: putting in the
Behrend function} for the action of a related group.}.

The situation for nodal fibers is a little different because of the
presence of the nodal point $z\in N$. The constructible function $\rho
_{*} (1)$, which is given by taking the Euler characteristic of the
fibers of the map
\[
\rho :\Hilb ^{1,\bullet ,\bullet }_{N} \left(\Xhat _{N\times E} \right) \to \Sym ^{\bullet }N,
\]
has the following form. Let $y_{1},\dotsc ,y_{l}$ be non-singular
points of $N$ and let $z\in N$ be the nodal point. Then $\rho ^{-1}
(bz+\sum a_{i}y_{i})$ parameterizes subschemes of $X$, supported on
$\Xhat _{N\times E}$, which have fixed curve support
\[
N\times \{x_{0} \}\cup \{z \}\times E \cup _{i}\{y_{i} \}\times E
\]
where the multiplicity along $\{z \}\times E$ is $b$ and the
multiplicity along $\{y_{i} \}\times E$ is $a_{i}$. Such a subscheme
is determined by its restriction to the formal neighborhoods $\Xhat
_{\{z \}\times E}$, $\Xhat _{\{y_{1} \}\times E},\dotsc , \Xhat
_{\{y_{l}\}\times E }$ and their complement $U$. The contribution of
the Euler characteristic of $U$ is given by
\[
(1-p)^{-e (N^{\circ })} = (1-p)^{l}
\]
where $N^{\circ }=N-\{z,y_{1},\dotsc ,y_{l} \}$. Therefore we see that
\[
\rho _{*} (1) (bz+\sum a_{i}y_{i}) = N (b)\prod _{i=1}^{l}F (a_{i})
\]
where $F (a)$ is as in section~\ref{sec: computing DThat for h=0}, and 
\[
N (b) = e\left(\Hilb ^{1,b,\bullet }\left(\Xhat _{\{z \}\times E}
\right) \right)
\]
where 
\[
\Hilb ^{1,b,n}\left(\Xhat _{\{z \}\times E} \right)\subset \Hilb ^{1,b,n} (X)
\]
is the sublocus parameterizing subschemes $Z$ whose curve support is
given by the union of $N\times \{x_{0} \}$ and a $b$-fold thickening
of $\{z \}\times E$ and such that all embedded points are supported on
$\Xhat _{\{z \}\times E}$.

So pushing the integral to $\Sym ^{\bullet }N$ and applying
lemma~\ref{lem: euler chars of sym} we get
\begin{align*}
\int _{\Hilb ^{1,\bullet ,\bullet }_{N}\left(\Xhat _{N\times E} \right)}1\,de &= \int _{\Sym ^{\bullet }N}\rho _{*} (1) de\\
&=\int _{\Sym ^{\bullet } (N-\{z \})} \prod _{i}F (a_{i})\,de \cdot \int _{\Sym ^{\bullet } (\{z \})} N (b)\,de\\
&=\left(\sum _{a=0}^{\infty } F (a)q^{a} \right)^{e (N-\{z \})} \cdot \left(\sum _{b=0}^{\infty }N (b)q^{b} \right).
\end{align*}
Note that $e (N-\{z \})=0$ so that the $F (a)$ term doesn't contribute. 

We compute the $N (b)$ contribution by using the $\left(\CC ^{*}
\right)^{2}$ action on 
\[
\Xhat _{\{z \}\times E}\cong \Spec (\CC
[[u,v]])\times E
\]
and arguing as in section~\ref{sec: computing DThat
for h=0}. We find
\begin{align*}
\sum _{b=0}^{\infty }N (b)q^{b}
&=\sum _{b=0}^{\infty }e\left(\Hilb ^{1,b ,\bullet }\left(\Xhat _{\{z \}\times E} \right) \right)q^{b}\\
&=\sum _{b=0}^{\infty }\sum _{\beta\vdash b } e\left(\Hilb ^{1,\beta ,\bullet }\left(\Xhat _{\{z \}\times E} \right) \right)q^{b}\\
&=\sum _{\beta } q^{|\beta |} \frac{\Vsf _{(1) (1)\beta } (p)}{\Vsf _{\emptyset \emptyset \beta } (p)}.
\end{align*}
We see that fact that the curve $N$ has a node is manifest in the term
in the numerator: the vertex $\Vsf _{(1) (1)\beta } (p)$ is counting
curve configurations which are locally monomial at the nodal point
$\{z \}\times \{x_{0} \}$ where the curve is degree 1 along the two
branches of the node and has the monomial thickening given by $\beta $
along the $E$ direction.

Putting this and the earlier computations together, we find that the
total contribution of the components with vertical curves is given by
the following:
\begin{align*}
e\left(\Hilb ^{1,\bullet ,\bullet }_{\Vert ,\fix } (X) \right)=&
-22\prod _{m=1}^{\infty } (1-q^{m})^{-24} \quad +\,\,\,24\prod _{m=1}^{\infty }
(1-q^{m})^{-23}\cdot \sum _{\beta } q^{|\beta |} \frac{\Vsf _{(1)
(1)\beta } (p)}{\Vsf _{\emptyset ,\emptyset ,\beta } (p)}\\
=&\quad 24 \prod _{m=1}^{\infty } (1-q^{m})^{-24}\left\{\frac{1}{12}
-1 +\prod _{m=1}^{\infty } (1-q^{m})\sum _{\beta }q^{|\beta |}
\frac{\Vsf _{(1) (1)\beta } (p)}{\Vsf _{\emptyset ,\emptyset ,\beta }
(p)} \right\}.
\end{align*}

\begin{proposition}
The following identity holds:
\[
\prod _{m=1}^{\infty } (1-q^{m})\sum _{\beta }q^{|\beta |} \frac{\Vsf
_{(1) (1)\beta } (p)}{\Vsf _{\emptyset \emptyset \beta } (p)} =
1+\frac{p}{(1-p)^{2}} + \sum _{d=1}^{\infty }\sum _{k|d} k
(p^{k}+p^{-k}) q^{d}.
\]
\end{proposition}
\emph{Proof sketch:} Using the Okounkov-Reshetikhin-Vafa formula for
the topological vertex \cite[Eqn~3.20]{Okounkov-Reshetikhin-Vafa}, and
some standard combinatorics, one can rewrite the left hand side of the
above equation so that it is given in terms of Bloch-Okounkov's
2-point correlation function \cite[Eqn~5.2]{Bloch-Okounkov}. Namely,
one can show that it is given by $1-F (t_{1},t_{2})$ in the limit
where $t_{1}$ and $t_{2}$ approach $p$ and $p^{-1}$ respectively. The
limit can be evaluated explicitly using \cite[Thm~6.1]{Bloch-Okounkov}
and this leads to the right hand side of the formula. Details can be
found in \cite{Bryan-Kool-Young}.

Plugging in the proposition's formula into the previously obtained
equation, we see that the non-diagonal contribution to $\DThat _{1}
(X)$ is given as
\[
e\left(\Hilb ^{1,\bullet ,\bullet }_{\Vert ,\fix } (X) \right) =
24\prod _{m=1}^{\infty } (1-q^{m})^{-24}\left\{\frac{1}{12}
+\frac{p}{(1-p)^{2}} +\sum _{d=1}^{\infty }\sum _{k|d}k
(p^{k}+p^{-k})q^{d} \right\}
\]

\subsection{Diagonal contributions.}

To finish our computation of $\DThat _{1} (X)$, it remains to compute
$e\left(\Hilb _{\diag ,\fix }^{1,\bullet ,\bullet } (X) \right)$.

Let $C\subset X$ be a diagonal curve. The projections onto the factors
of $X=S\times E$ induce maps
\begin{align*}
p_{S}:&C\to F_{y}\\
p_{E}:&C\to E
\end{align*}
where $F_{y} $ is a fiber curve, and the maps have degree 1 and some $d>0$ respectively. $F_{y}$ cannot be a nodal fiber since then $C$ would have geometric genus 0 and consequently it would not admit a non-constant map to $E$. The above maps induce a map
\[
f:F_{y}\to E
\]
which must be unramified by the Riemann-Hurwitz formula. Thus the
diagonal curve $C$ is contained in the surface $F_{y}\times E$ and is
given by the graph of the map $f$. Recall that we fixed a slice for
the $E$ action on $\Hilb ^{1,d,n}_{\diag } (X)$ by requiring that the
diagonal curve meets $S_{x_{0}}$ at the section; this is equivalent to
requiring that $f (s)=x_{0}$ where $s\in F_{y} $ is the section point
on $F_{y}$. Up to automorphisms, such a map $f$ must be a group
homomorphism of the corresponding elliptic curves. Assuming that $E$
is generic, so that the only non-trivial automorphism is given by
$x\mapsto -x$, we see that every diagonal curve (with the fixed
condition) is of the form
\[
\{(z,f (z))\in F_{y}\times E \}\text{  or  } \{(z,-f (z))\in F_{y}\times E \}
\]
where $f:F_{y}\to E $ is a group homomorphism.

The number of group homomorphisms of degree $d$ to a fixed elliptic
curve $E$ is given by $\sum _{k|d}k$. This classical fact can be seen
by counting index $d$ sublattices of $\ZZ \oplus \ZZ $. For each such
cover, $F\to E$, the domain elliptic curve will occur exactly 24 times
in the fibration $S\to \PP ^{1}$. So we find that the total number of
diagonal curves having degree $d$ in the $E$ direction is
\[
2\cdot 24\sum _{k|d}k.
\]
Each such diagonal curve can be accompanied by horizontal curves (with
thickenings) as well as embedded points. The contribution of these
components of the Hilbert scheme is computed in exactly the same way
as the contribution of the curves with a smooth vertical component
$F$. Recall that $e\left(\Hilb ^{1,\bullet, \bullet }_{F} (X)
\right)=\prod _{m=1}^{\infty } (1-q^{m})^{-24}$. Taking into account
the degree of the diagonal curves, we thus find
\[
e\left(\Hilb ^{1,\bullet ,\bullet }_{\diag ,\fix } (X) \right) =\prod
_{m=1}^{\infty } (1-q^{m})^{-24} \left(2\cdot 24\cdot \sum
_{d=1}^{\infty }\sum _{k|d}kq^{k} \right).
\]

Finally, adding the vertical and diagonal contributions together we
arrive at
\[
\DThat _{1} (X) = 24q^{-1}\prod _{m=1}^{\infty }
(1-q^{m})^{-24}\left\{\frac{1}{12} +\frac{p}{(1-p)^{2}} +\sum
_{d=1}^{\infty }\sum _{k|d}k (p^{k}+p^{-k}+2)q^{d} \right\}.
\]

Note that this formula is off from the desired formula for $\DT _{h=1}
(X)$ by an overall minus sign and a minus sign on the 2. In fact we
will see in section~\ref{sec: putting in the Behrend function} that
due to the Behrend function, the contribution of the diagonal
components carry the opposite sign of the contribution of the vertical
components. Denoting the contribution to $\DThat_{1} (X)$ coming from
$\Hilb^{1,\bullet ,\bullet }_{\Vert ,\fix } (X) $ and from
$\Hilb^{1,\bullet ,\bullet }_{\diag ,\fix } (X) $ by $\DThat _{1,\Vert
} (X)$ and $\DThat _{1,\diag } (X)$ respectively, we find that we need
to show
\[
\DT _{1} (X) = -\DThat _{1,\Vert } (X) + \DThat _{1,\diag } (X).
\]

\section{Putting in the Behrend function}\label{sec: putting in the
Behrend function}

The goal of this section is to prove, assuming Conjecture~\ref{conj:
Behrend fnc conj}, the relations
\begin{align*}
\DT_{0}(X)&= -\DThat_{0}(X)\\
\DT_{1}(X)&=-\DThat_{1,\Vert}(X) + \DThat_{1,\diag}(X)
\end{align*}
which is all that is needed to complete the proof of Theorem~\ref{thm:
main theorem: DTh for h=0 and h=1}.

\subsection{Overview}
Our general strategy for computing $\DThat (X)$,
the unweighted Euler characteristics of the Hilbert schemes, utilized
the following general scheme.
\begin{enumerate}
\item Using the geometric support of curves (and/or points) of the
subschemes, we stratified $\Hilb (X)$ such that the strata could be
written as products of simpler Hilbert schemes.
\item We utilized actions of $\CC ^{*}$ or $E$ which could be defined
on individual factors in the stratification to discard strata not
fixed by the action and restrict to fixed points.
\item We found that some strata were parameterized by symmetric
products, and we pushed forward the Euler characteristic computation
to the symmetric products where we used Lemma~\ref{lem: euler chars of
sym}.
\item After possibly iterating steps (1)--(3), we reduced the
computation to counting discrete subscheme configurations, namely
those which are given formally locally by monomial ideas. These we
counted with the topological vertex.
\end{enumerate}

In order to incorporate the weighting by $\nu$, the Behrend function,
into the Euler characteristics in the above strategy, we need to
modify steps (2) and (4).

For (2) to apply to $\nu$-weighted Euler characteristics, we need to
know that $\nu$, restricted to the relevant stratum $S\subset \Hilb
(X)$, is invariant under the action of the group. We can do this by
showing that the group (or possibly a modification of the group) acts
on the formal neighborhood of $S$ inside of $\Hilb (X)$. This works
since the value of the Behrend function at a point only depends on the
formal neighborhood of the point.

To modify step (4), the final step, we will need to know the value of
the Behrend function at subschemes given formally locally by monomial
ideals. In particular, we want to show the value is given by $\pm 1$,
where the sign alternates as $n$ increases. This will account for the
relatively simple relationship between $\DT (X)$ and $\DThat (X)$.

We are only partially able to succeed with the above modification. In
the first two iterations of steps (1), (2), and (3) of the strategy,
we succeed in extending the actions of the group (or related group) to
the formal neighborhood of the strata. However, in the final iteration
of (1)-(3), the $(\CC^{*})^{3}$ action we used on the strata does not
obviously extend to their formal neighborhoods (quite possibly such an
extension does not exist). To remedy this situation, we must assume a
conjecture first formulated in \cite[Conj~18]{Bryan-Kool}.

\subsection{Elaboration.}  Let us expand on the above discussion to
highlight the issues.

In the first two iterations of steps (1), (2), and (3), we stratified
by the curve support of the subschemes. These strata decompose into
the product given by the equations \eqref{eqn:
Hilb/E=Hilb(X-CxE)*Hilb(Xhat)}, \eqref{eqn:
Hilb/E=Hilb(X-FxE)*Hilb(Xhat)}, and \eqref{eqn:
Hilb/E=Hilb(X-NxE)*Hilb(Xhat)} where the second factors correspond to
connected components of the curve with pure vertical support. This
factor admits an extra $E$ action, and we wish to show that this
action extends to the formal neighborhoods of the strata.

The formal neighborhood of a closed point in $\Hilb (X)$ parameterizes
infinitesimal deformations of the subscheme corresponding to the
closed point. Connected components of a subscheme deform independently
from each other and consequently the products given in equations
\eqref{eqn: Hilb/E=Hilb(X-CxE)*Hilb(Xhat)}, \eqref{eqn:
Hilb/E=Hilb(X-FxE)*Hilb(Xhat)},  and \eqref{eqn:
Hilb/E=Hilb(X-NxE)*Hilb(Xhat)} extend to their formal neighborhoods,
as do the $E$ actions.

The first factor in the products \eqref{eqn:
Hilb/E=Hilb(X-CxE)*Hilb(Xhat)}, \eqref{eqn:
Hilb/E=Hilb(X-FxE)*Hilb(Xhat)}, and \eqref{eqn:
Hilb/E=Hilb(X-NxE)*Hilb(Xhat)} correspond to subschemes supported in
$\Xhat_{C\times E}$ where $C$ is the curve in the slice $S\times
x_{0}$, which is either $C_{0}$ (for $h=0$), or $N$ or $F$ (for
$h=1$). This stratum was further stratified by the location of the
vertical components of such curves. Each of these strata admits a
$\CC^{*}\times \CC^{*}$ induced by the action 
\[
\Xhat_{\{y \}\times
E}\cong \Spec \CC [[x,y]]\times E,
\]
the formal neighborhood of a vertical component. This action does
\emph{not} obviously extend to the formal neighborhood of the strata
in the Hilbert scheme. The basic issue is the following.

A subscheme in $X$ whose reduced support is 
\[
C\times x_{0} \cup y\times E
\]
and whose multiplicies along $C\times x_{0}$ and $y\times E$ is 1 and
some $a\geq 1$ on $E$ respectively is uniquely determined by its
restrictions to 
\[
\Xhat_{y\times E} \quad \text{and}\quad  \Xhat_{C\times X_{0}}-\{y\times x_{0}  \}.
\]
The action of $\CC^{*}\times \CC^{*}$ on $\Xhat_{y\times E}$ thus
induces an action on this strata. However, it is not clear if the
action on this stratum extends to an action on its formal neighborhood
due to (for example) the possibility of obstructed infinitesimal
deformations of the subscheme smoothing the node.

We can circumvent this problem by using somewhat different group
actions. In the case of $h=0$, the $\CC^{*}\times \CC^{*}$ action on
$\Xhat_{y\times E}$ \textit{does} extend to an action on
$\Xhat_{C_{0}\times E}$. The reason is that $C_{0}$ is a smooth $-2$
curve in a surface and consequently its formal neighborhood in the
surface is isomorphic to the formal neighborhood of the zero section
in the total space of $T^{*}\PP^{1}$. This formal scheme carries an
$\CC^{*}\times \CC^{*}$ action which can be choosen to be compatible
with the one on $\Xhat_{y\times E}$. This then induces an action on
$\Hilb (\Xhat_{C_{0}\times E})/E$ which naturally extends to its formal
neighborhood in $\Hilb (X)/E$. 

In the case of $h=1$, we can also construct global group actions on
$\Xhat_{C\times E}$, but different from the $\CC^{*}\times \CC^{*}$
action previously used. For $C=F$, a smooth elliptic curve, we
previously noted that the action of $F$ on $F\times E$ given by
translation on the first factor, does not extend to $\Xhat_{F\times
E}$. The reason is that the linear system of $F$ in $S$ induces a
non-trivial elliptic fibration $S\to \PP^{1}$ and thus a non-trivial
fibration $\hat{S}_{F}\to \Spec \CC [[t]]$. However, after choosing a
section of the above map, $\hat{S}_{F}$ is a group scheme over $\Spec
\CC [[t]]$ and the Mordell-Weil group of sections acts (freely) on
$\hat{S}_{F}$, and thus on $\Xhat_{F\times E} = \hat{S}_{F}\times
E$. This induces an action of the Mordell-Weil group on $\Hilb
(\Xhat_{F\times E})/E$ which extends to its formal neighborhood and is
free whenever the degree in the vertical direction is non-zero. The
Mordell-Weil group is an an extension of the group $F$ by an additive
group and so its orbits have Euler characteristic zero. Consequently
the conclusion expressed by equation~\eqref{eqn:
e(Hilb(Xhat_{FxE}))=1} for the usual Euler characteristics also
applies to the $\nu$-weighted Euler characteristics.

Similarly, the Mordell-Weil group of sections of $\hat{S}_{N}\to \Spec
\CC [[t]]$ acts on $\Hilb (\Xhat_{N\times E})/E$ and its formal
neighborhood. This group is an extension of $\CC^{*}$ (the group of
the nodal fiber) by an additive group and hence necessarily
splits. Thus we get an action of $\CC^{*}$ on $\Hilb (\Xhat_{N\times
E})/E$ which is compatible with the Behrend function. The fixed points
are subschemes whose vertical part has support in $\Xhat_{\{z \}\times
E}$ where $z\in N$ is the node. Moreover, the $\CC ^{*}$ action here
is given by $\lambda (x,y) = (\lambda x, \lambda^{-1}y)$ for a
suitable choice of isomorphism $\Xhat_{\{z \}\times E} \cong E\times
\Spec \CC [[x,y]]$. The fixed subschemes under this action correspond
to subschemes whose maximal Cohen-Macaulay subscheme is formally
locally given by monomial ideals.

Via the above, we have shown that the $\nu$-weighted Euler
characteristics of the Hilbert schemes are equal to the $\nu$-weighted
Euler characteristics of the locus in the Hilbert scheme
parameterizing subschemes whose maximal Cohen-Macaulay subscheme is
given formally locally by monomial ideals.

In the final iteration of steps (1)-(3), we were able to further
reduce to subschemes whose embedded points are also local monomial. To
achieve this we further stratified our strata by the location of the
embedded points. We then used the action of $\left(\CC^{*}
\right)^{3}$ on the formal neighborhoods of the embedded points to
construct a $\left(\CC^{*} \right)^{3}$ action on these
substrata. Unfortunately, we are unable to show that this action on
the substrata extends to a formal neighborhood of the
substrata. Consequently, we cannot show that the Behrend function is
compatible with the action.

To circumvent this problem, we will assume the validity of the
following conjecture, restated from \cite[Conj~18]{Bryan-Kool}.  Let
$Y$ be any quasi-projective Calabi-Yau threefold.  Let $C\subset Y$ be
a (not necessarily reduced) Cohen-Macaulay curve with proper
support. Assume that the singularities of $C_{\red}$ are locally
toric\footnote{This means that formally locally $C_{\red}$ is either
smooth, nodal, or the union of the three coordinate axes. That is at
$p\in C_{\red}\subset Y$ the ideal $\widehat{I}_{C_{\red}}\subset
\widehat{\mathcal{O}}_{Y,p}$ is given by $(x_{1},x_{2})$,
$(x_{1},x_{2}x_{3})$, or $(x_{1}x_{2},x_{2}x_{3},x_{1}x_{3})$ for some
isomorphism $\widehat{\mathcal{O}}_{Y,p}\cong \CC
[[x_{1},x_{2},x_{3}]]$. }. Let
\[
\Hilb^{n}(Y,C) = \{Z\subset Y \text{ such that $C\subset Z$ and
$I_{C}/I_{Z}$ has finite length $n$} \}.
\]

Note that $\Hilb^{n}(Y,C)\subset \Hilb (Y)$ and let $\nu$ denote the
Behrend function on $\Hilb (Y)$.

\begin{conjecture}[\cite{Bryan-Kool}]\label{conj: Behrend fnc conj}
\[
\int_{\Hilb^{n}(Y,C)} \nu \, de = (-1)^{n} \nu ([C]) \int_{\Hilb^{n}(Y,C)} \, de,
\]
where $\nu ([C])$ is the value of the Behrend function at the point $[C]\in \Hilb (Y)$.
\end{conjecture}

This conjecture allows us to promote the remaining part of our Euler
characteristic computation to $\nu$-weighted Euler characteristic,
once we compute the value of the Behrend function at locally monomial
Cohen-Macaulay subschemes. Suppose we could show that value was always
$(-1)^{n}$ where $n$ is the holomorphic Euler characteristic of the
subscheme, then, because of the $(-p)$ built directly into the
definition of $\DT_{g}(X)$, we could conclude that $\DT_{g}(X) =
\DThat_{g}(X)$ for $h=0$ and $h=1$. We will show this is nearly true:
in fact the Behrend function has value $-(-1)^{n}$ for all curves in
the $h=0$ case and for those without diagonal components in the $h=1$
case. Curves with diagonal components have the sign $(-1)^{n}$.

\subsection{The Behrend function at Cohen-Macaulay subschemes.}

The only Cohen-Macaulay subschemes which contribute to the invariants
in our counting scheme are all of the following form:
\begin{enumerate}
\item $Z\times E$,
\item $C_{0}\times x_{0}\cup Z\times E \quad \quad (h=0),$
\item $N\times x_{0}\cup Z\times E \quad \quad (h=1),$ and
\item $C$, a diagonal curve (necessarily reduced)
\end{enumerate}
where $Z\subset S$ is a zero dimensional subscheme of length $d$. For
each of these cases, we show these subschemes lie in a smooth, open
locus of the corresponding component of the Hilbert scheme and hence
the value of Behrend function is given by $(-1)^{D}$ where $D$ is the
dimension of that open set.

This is entirely parallel to the analysis in sections 8 and 9 of
\cite{Bryan-Kool}. Namely, we construct an explicit $D$-dimensional
family of such subschemes and we then show that the family is smooth
by showing the Zariski tangent space also has dimension $D$.

\begin{proposition}\label{prop: dimen of families of CM schemes}
The following families of subschemes have the given dimensions and are
open sets of Hilbert scheme which are smooth at subschemes given
locally by monomial ideals.
\begin{enumerate}
\item The locus of schemes of the form $Z\times E$ has dimension $2d$,
where $Z\subset S$ is a length $d$ zero-dimensional subscheme.
\item  The locus of schemes of the form $C_{0}\times x_{0}\cup Z\times
E$ has dimension $2d-k$, where $Z\subset S$ is a length
$d$ zero-dimensional subscheme such that $\operatorname{length}(C_{0}\cap Z)=k$.
\item The locus of schemes of the form $C\times x_{0}\cup Z\times
E$ has dimension $2d-k+1$, where $Z\subset S$ is a length
$d$ zero-dimensional subscheme such that $\operatorname{length}(C\cap
Z)=k$ and $C\in |F|$ is any curve in the class $F$ (including $N$).
\item The locus of diagonal curves has dimension 0.
\end{enumerate}
\end{proposition}
\begin{proof}
The family (1) is paramterized by $\Hilb^{d}(S)$ which has dimension
$2d$. The family (2) is parameterized by the locus of points $[Z]\in
\Hilb^{d}(S)$ given by the codimension condition that
$\operatorname{length}(Z\cap C_{0})=k$. This is smooth of dimension
$2d-k$ by \cite[Thm~20]{Bryan-Kool}. The family (3) is parameterized
by the locus of points $[Z]\in \Hilb^{d}(S)$ such that
$\operatorname{length}(Z\cap C)=k$ for some $C\in |F|$. This family
maps to $|F|\cong \PP^{1}$ with fibers of dimension $2d-k$ (again by
\cite[Thm~20]{Bryan-Kool}).

To complete the proof of the proposition, we want to show these sets are
open and smooth at the monomial ideals. It suffices to show that the
dimension of the Zariski tangent space is of the given dimension. The
tangent space of the Hilbert scheme at a given subscheme $Z$ is given
by the group $ \Hom
(I_{Z},\O_{Z}) \cong \operatorname{Ext}^{1}(I_{Z},I_{Z})$. These Ext
groups can be computed at monomial ideals using the technique of
\cite[\S~9]{Bryan-Kool}. Indeed, the proof of Thm~21 in
\cite{Bryan-Kool} applies with minor modifications to the cases
(1)-(3). Finally it is easy to see that a diagonal curve $C\subset X$
is scheme-theoretically isolated up to translation by $E$. For
example, since $C$ is smooth and reduced, the Zariski tangent space is
given by $H^{0}(C,N_{C/X})$. Here $C$ is given as the graph of some
homomorphism $f:F_{y}\to E$ and thus $N_{C/X}$ is given as the non-trivial
extension of $f^{*}N_{F/S}\cong \O_{C} $ by $N_{C/F_{y}\times E}\cong
\O_{C}$. Therefore $H^{0}(C,N_{C/X})\cong \CC$ which corresponds to
the translations by $E$.
\end{proof}

Using the normalization exact sequence, one easily computes $n$, the
holomorphic Euler characteristic of the subschemes given by the four
families in Proposition~\ref{prop: dimen of families of CM
schemes}. Namely, 
\[
n=\begin{cases}
0&\text{for family (1)}\\
1-k&\text{for family (2)}\\
-k&\text{for family (3)}\\
0&\text{for family (4)}\\
\end{cases}
\]
Since the value of the Behrend function at a smooth point of dimension
$D$ is $(-1)^{D}$, the above formulas, along with
Proposition~\ref{prop: dimen of families of CM schemes}  gives us
\[
\nu  =
\begin{cases}
(-1)^{2d}=1=(-1)^{n}&\text{for family (1)}\\
(-1)^{2d-k}=(-1)^{-k}=-(-1)^{n}&\text{for family (2)}\\
(-1)^{2d-k+1}=(-1)^{-k+1}=-(-1)^{n}&\text{for family (3)}\\
(-1)^{0}=1=(-1)^{n}&\text{for family (4)}\\
\end{cases}
\]

In the $h=0$ case, the locally monomial Cohen-Macaulay subschemes
which contribute to the $\nu$-weighted Euler characteristics are
disjoint unions of curves in family (1) and a single curve in family
(2). Thus, they always come with the sign $-(-1)^{n}$ and we can
conclude
\[
\DT_{0}(X) = -\DThat_{0}(X).
\]
In the $h=1$ case, the locally monomial Cohen-Macaulay subschemes
without a diagonal component which contribute to the $\nu$-weighted
Euler characteristics are disjoint unions of curves in family (1) with
a single curve in family (3). Thus the contribution of the
non-diagonal curves to $\DT_{h=1}(X)$ is given by
$-\DThat_{1,\Vert}(X)$.

Finally, locally monomial Cohen-Macaulay
subschemes with a diagonal curve which contribute to the
$\nu$-weighted Euler characteristic are disjoint unions of curves in
family (1) with a single curve in family (4). Thus the contribution of
the diagonal curves to $\DT_{h=1}(X)$ is given by $\DT_{1,\diag}(X)$.

Putting it all together we see that 
\begin{align*}
\DT_{0}(X)&= -\DThat_{0}(X)\\
\DT_{1}(X)&=-\DThat_{1,\Vert}(X) + \DThat_{1,\diag}(X)
\end{align*}
as desired which completes the proof of Theorem~\ref{thm: main theorem: DTh for h=0 and h=1}.

\section{Prospects for $h>1$}
   
Our strategy can be applied to the case of computing $\DT _{h} (X)$
for $h>1$ although some new issues and complexities aries. Our method
is predicated on two main things:
\begin{enumerate}
\item Having a detailed understanding of the possible curve support of
subschemes in the class $\beta _{h}+dE$.
\item Having the singularities of the curves be formally locally toric
so that vertex methods can be applied.
\end{enumerate}

Addressing issue (1) grows increasingly difficult as $h$ gets
larger. For relatively small values of $h$, one has a pretty explicit
understanding of the curves in the linear system of $\beta _{h}$. To
address (1) fully also requires understanding ``diagonal''
curves. This amounts to solving the following interesting enumerative
question about $K3$ surfaces:
\begin{question}
Given a $K3$ surface with a irreducible curve class $\beta $ of square
$2h-2$, how many curves of geometric genus $g$ are in the class $\beta
$ which admit a degree $d$ map to a (fixed) elliptic curve $E$?
\end{question}

Note that genus $g$ curves on a $K3$ surface always move in an $g$
dimensional family, and the dimension of genus $g$ curves admitting a
degree $d$ map to an elliptic curve $E$ is $2g-3$ (independent of $d$)
and thus is codimension $g$ in $\overline{M} _{g}$. Therefore this is
a dimension zero problem.

Addressing issue (2) requires some new ideas. Starting at $h=2$, one
must confront curves with singularities worse than nodes. For small
$h$, one should be able to finesse around this issue. For example, for
$h=2$, one will need the contribution of a curve in $K3$ with a cusp,
with a $d$-fold thickening of $E$ attached at the cusp. This is not
locally toric and so its contribution cannot be computed using the
vertex methods that we used for locally toric subschemes. However,
this contribution can be fully determined from the $h=1$ results as
follows. One redoes the $h=1$ computation using an elliptically
fibered $K3$ which has a cuspidal singular fiber. This will enable one
to reverse engineer the cusp contribution which one can then apply to
compute the $h=2$ case fully. However, it isn't clear how far one can
get with this inductive sort of strategy.

A more satisfying way to handle a contribution from arbitrary surface
singularity would be to relate this contribution to the knot
invariants of the link of the singularities. This would be in keeping
with the work of Shende and Oblombkov \cite{Oblomkov-Shende-Duke}
although it doesn't appear that their results direct apply.

\bibliography{/Users/jbryan/jbryan/resources/mainbiblio}
\bibliographystyle{plain}

\end{document}